\documentclass[12pt]{article}

\usepackage{amssymb}%
\usepackage{amsmath}%
\usepackage{amsthm}%
\usepackage{xcolor}%

\allowdisplaybreaks%

{\theoremstyle{plain}%
 \newtheorem{theorem}{Theorem}

}
{\theoremstyle{remark}

}
{\theoremstyle{definition}

}

\begin{document}

\begin{center}
{\Large Extensions of the truncated pentagonal number theorem}

 \ 

{\textsc{John M. Campbell}} 

 \ 

\end{center}

\begin{abstract}
 Andrews and Merca introduced and proved a $q$-series expansion for the partial sums of the $q$-series in Euler's pentagonal number 
 theorem. Kolitsch, in 2022, introduced a generalization of the Andrews--Merca identity via a finite sum expression for $ \sum_{n 
 \geq k} \frac{ q^{ (k +m) n } }{ \left( q; q \right)_{n} } \left[ \begin{smallmatrix} n - 1 \\ k - 1 \end{smallmatrix} \right]_{q}$ for positive 
 integers $m$, and Yao also proved an equivalent evaluation for this $q$-series in 2022, and Schlosser and Zhou extended this result 
 for complex values $m$ in 2024, with the $m = 1$ case yielding the Andrews--Merca identity, and with the $m = 2$ case having been 
 proved separately by Xia, Yee, and Zhao. We introduce and apply a method, based on the $q$-version of Zeilberger's algorithm, that 
 may be used to obtain finite sum expansions for $q$-series of the form $ \sum_{n \geq 1} \frac{ q^{ p(k) n } }{ \left( q; q \right)_{n + 
 \ell_2} } \left[ \begin{smallmatrix} n - \ell_{1} \\ k - 1 \end{smallmatrix} \right]_{q} $ for linear polynomials $p(k)$ and $\ell_{1} \in 
 \mathbb{N}$ and $\ell_{2} \in \mathbb{N}_{0}$, thereby generalizing the Andrews--Merca identity and the Kolitsch, Yao, and 
 Schlosser--Zhou identities. For example, the $(p(k), \ell_1, \ell_2) = (k+1, 2, 0)$ case provides a new truncation identity for Euler's 
 pentagonal number theorem. 
 \end{abstract}


\vspace{0.1in}


\section{Introduction}
 Let the \emph{$q$-Pochhammer symbol} be such that $(a;q)_{n} = (1 - a) (1 - aq) (1 - a q^2) \cdots (1 - a q^{n-1})$ for a 
 nonnegative integer $n$, and with $(a;q)_{\infty} = \prod_{k=0}^{\infty} (1 - a q^{k})$. One of the most fundamental results within the 
 areas of mathematics related to integer partitions is given by \emph{Euler's pentagonal number theorem}, which can be formulated 
 so that 
\begin{equation}\label{displayEuler}
 (q;q)_{\infty } = 1 + \sum_{n=1}^{\infty} (-1)^{n} q^{\frac{n(3n+1)}{2}} \left( 1 + q^{-n} \right). 
\end{equation}
 We find that this is equivalent to 
\begin{equation}\label{equivEuler}
 (q;q)_{\infty } = \sum_{n = 0}^{\infty} (-1)^{n} q^{\frac{n(3n+1)}{2}} \left( 1 - q^{2n+1} \right), 
\end{equation}
 by rewriting $\sum _{n=0}^m (-1)^n q^{\frac{n (3 n+1)}{2}} \left(q^{-n}+q^{2 n+1}\right) $ as $ 1+(-1)^m q^{\frac{(m+1) (3 m + 
 2)}{2}}$. A remarkable result due to Andrews and Merca \cite{AndrewsMerca2012} provides a $q$-series expansion for a truncated 
 version of \eqref{equivEuler}. This leads us to consider the development of techniques to extend the Andrews--Merca 
 truncation identity. 

 \emph{Gaussian binomial coefficients} or \emph{$q$-binomial coefficients} may be defined so that 
 \[ \left[ \begin{matrix} n \\ k \end{matrix} \right]_{q} = \begin{cases} 
 \frac{ \left( q; q \right)_{n} }{ \left( q;q \right)_{k} 
 \left( q;q \right)_{n-k} } & \text{if $0 \leq k \leq n$,} \\ 
 0 & \text{otherwise}.
 \end{cases} \] 
 The main result from the work of Andrews and Merca \cite{AndrewsMerca2012} on the truncation of the Euler pentagonal number 
 theorem is given by the $q$-identity such that $$ \frac{1}{ \left( q; q \right)_{\infty} } \sum_{j=0}^{k-1} (-1)^j q^{\frac{j(3 j + 
 1)}{2}} \left( 1 - q^{2j+1} \right) = 1 + (-1)^{k-1} \sum_{n = 1}^{\infty} \frac{ q^{ \binom{k}{2} + (k + 1) n } }{ \left( q;q \right)_{n} } 
 \left[ \begin{matrix} n - 1 \\ k - 1 \end{matrix} \right]_{q}. $$ This Andrews--Merca identity has been influential in many areas of 
 combinatorics and number theory, and has been generalized in a number of ways. In this direction, a notable generalization was proved 
 in 2022 by Kolitsch \cite[Theorem 2]{Kolitsch2022} and is such that 
\begin{multline*}
 \frac{ \left( q^{k+1}; q \right)_{m-1} }{ \left( q;q \right)_{\infty} } 
 \sum_{j=0}^{k-1} (-1)^{j} q^{(k + m) j + \binom{j}{2} } \left( q^{j+1}; q \right)_{k - j} = \\ 
 1 + (-1)^{k-1} \sum_{n=k}^{\infty} \frac{ q^{\binom{k}{2} + (k +m) n } }{ \left( q; q \right)_{n} } 
 \left[ \begin{matrix} n - 1 \\ k - 1 \end{matrix} \right]_{q} 
\end{multline*}
 for positive integers $k$ and $m$. The equivalent identity such that 
\begin{multline*}
 \frac{1}{ \left( q;q \right)_{\infty} } 
 \sum_{j=0}^{k-1} (-1)^j q^{\frac{j(3j+2m-1)}{2}} 
 \left( q^{j+1};q \right)_{m-1} \left( 1 - q^{2j+m} \right) = \\ 
 1 + (-1)^{k-1} \sum_{n=k}^{\infty} \frac{ q^{\binom{k}{2} + (k + m) n } }{ \left( q; q \right)_{n} } 
 \left[ \begin{matrix} n - 1 \\ k - 1 \end{matrix} \right]_{q} 
\end{multline*}
 was also proved independently by Yao in 2022 \cite[Theorem 1.2]{Yao2022}. This was generalized in 2024 by 
 Schlosser and Zhou \cite[Corollary 1.8]{SchlosserZhou2024} for complex $m$. 
 Moreover, the equivalent version 
\begin{multline*}
 \frac{1}{ \left( q;q \right)_{\infty} } 
 \sum_{j=0}^{k-1} (-1)^{j} q^{\frac{3j(j+1)}{2}} 
 \left( 1 - q^{j+1} \right) \left( 1 - q^{2j+2} \right) = \\ 
 1 + (-1)^{k-1} \sum_{n = 1}^{\infty} \frac{ q^{\binom{k}{2} + (k + 2) n } }{ \left( q; q \right)_{n} } 
 \left[ \begin{matrix} n - 1 \\ k - 1 \end{matrix} \right]_{q} 
\end{multline*}
   of the $m = 2$ case of the above Kolitsch identity and the above Yao identity was highlighted as a main result in    a 2022 contribution   
   due to Xia, Yee, and Zhao \cite{XiaYeeZhao2022} and     as providing a truncation of an equivalent version of the Euler pentagonal number   
   theorem.     The truncation identities above     and many further results related to the Andrews--Merca truncation   
  \cite{AndrewsMerca2018,GuoZeng2013,KolitschBurnette2015,Mao2015,WangYee2019,Xia2022,Yee2015} motivate further       
    generalizations of this truncation.   

 The Kolitsch identity, the Yao identity, the Schlosser--Zhou identity, and the Xia--Yee--Zhao identity raise questions as to how 
 expressions of the form 
\begin{equation}\label{mainobject}
 1 + (-1)^{k-1} \sum_{n = 1}^{\infty} \frac{ q^{\binom{k}{2} + \rho(k) \, 
 n } }{ \left( q; q 
 \right)_{n + \ell_2} } \left[ \begin{matrix} n - \ell_1 \\ k - 1 \end{matrix} \right]_{q} 
\end{equation}
 could be expressed with truncated $q$-sums, for a linear polynomial $\rho(k)$ 
 and for $\ell_{1} \in \mathbb{N}$ and for $\ell_{2} \in \mathbb{N}_{0}$. 
 This provides the main purpose of our paper. 
 For $\ell_1 = 0$, series of the form indicated in \eqref{mainobject} 
 can often be reduced to classical results on basic hypergeometric series, and hence our 
 disregarding this case, and similarly for the $\ell_1, \ell_2 < 0$ cases. 
 
 Observe that for $\ell_1 = 1$ and $\ell_2 = 0$, the $\rho(k) = k + 1$ and $\rho(k) = k + m$ and $\rho(k) = k + 2$ cases of 
 \eqref{mainobject} provide, respectively, the $q$-series in the Andrews--Merca identity, the Kolitsch and Yao and Schlosser--Zhou 
 identities, and the Xia--Yee--Zhao identity. Since $\rho(k)$ is monic in all of these cases, this leads us to consider the problem of 
 determining a finite sum expansion of \eqref{mainobject} for $\rho(k) = 2 k + m$ and $\rho(k) = 3 k + m$ and $\rho(k) = 4 k + 
 m$, as in Section \ref{sectionInfinite} 
 below. To begin with, we highlight, as in Section \ref{sectionNew} below, 
 new $q$-series for truncations of 
  equivalent versions of  the  $q$-series involved in 
 the Euler identity in \eqref{displayEuler}. 

\section{New truncated versions of the pentagonal number theorem}\label{sectionNew}
 Informally, we write $F(n, k)$ in place of the product of $(-1)^{k-1}$ and a variant of the summand of the $q$-series in the above 
 formulation of the Andrews--Merca identity (where this $q$-series is over $n \in \mathbb{N}$), such as the summand displayed in 
 \eqref{mainobject}, and we require $F(n, k)$ to be $q$-hypergeometric. We then set $\mathcal{F}(n, k) := F(k, n)$, and we apply the 
 $q$-version of Zeilberger's algorithm to $\mathcal{F}(n, k)$, after inputting 
\begin{verbatim}
with(QDifferenceEquations):
\end{verbatim}
 into the Maple Computer Algebra System. In exceptional cases, this results in a second-order difference equation, writing $$ p_{1}(n) 
 \mathcal{F}(n+2,k) + p_{2}(n) \mathcal{F}(n+1,k) + p_{3}(n) \mathcal{F}(n, k) = \mathcal{G}(n, k + 1) - \mathcal{G}(n, k) $$ for 
 $q$-polynomials $p_{1}(n)$, $p_{2}(n)$, and $p_{3}(n)$ and for a $q$-hypergeometric function $\mathcal{G}(n, k)$ such that 
 $ \mathcal{G}(n ,k) = \mathcal{R}(n, k) \mathcal{F}(n, k)$ for a $q$-rational function $\mathcal{R}(n, k)$. In exceptional cases, we 
 obtaining a vanishing expression on the right-hand side of the above difference equation, via a telescoping phenomenon from the 
 application of an infinite summation operator $\sum_{k} \cdot$ over a given index set, i.e., so that we obtain a recursion 
\begin{equation}\label{frecursion}
 p_{1}(k) f(k+2) + p_{2}(k) f(k+1) + p_{3}(k) f(k) = 0 
\end{equation} 
 for $f(k) = \sum_{n} F(n, k)$. In exceptional cases, the recurrence in \eqref{frecursion} can be solved explicitly as a finite sum, using 
 computer algebra tools as in the {\tt RSolve} command in the Wolfram Mathematica system. If the required evaluations for the base 
 cases can then be determined, this provides the desired finite sum evaluation for $\sum_{n} F(n, k)$. 

 As a natural variant of the $q$-series involved in the Andrews--Merca identity, we consider the $q$-series obtained by replacing the 
 $q$-binomial coefficient $ \left[ \begin{smallmatrix} n - 1 \\ k - 1 \end{smallmatrix} \right]_{q} $ with $ \left[ \begin{smallmatrix} n - 
 2 \\ k - 1 \end{smallmatrix} \right]_{q}$, noting that $ \left[ \begin{smallmatrix} n - 2 \\ k - 1 \end{smallmatrix} \right]_{q} = 
 \frac{1-q^{n-k}}{1-q^{n-1}} \left[ \begin{smallmatrix} n - 1 \\ k - 1 \end{smallmatrix} \right]_{q}$. We have applied out above method 
 to derive and prove the following result, which provides a new $q$-series expansion for a truncation of an equivalent version of the 
 series involved in Euler's pentagonal number theorem. 

\begin{theorem}\label{naturalreplace}
 The truncation identity 
\begin{multline*}
 \sum_{j=0}^{k-1} (-1)^j q^{\frac{j(3j-1)}{2}} \left( 1 + q^{j} \right) = 
 1 - \frac{ \left( -1 \right)^{k} q^{\frac{k(3k-1)}{2}} }{ 1 - q^{k} + q^{2k} } 
 + \\ \left( q;q \right)_{\infty} 
 \left( 1 - \frac{ (1-q) \left( -q \right)^{k} }{q \left( 1 - q^{k} + q^{2k} \right) } 
 \sum_{n = 2}^{\infty} \frac{ q^{\binom{k}{2} + (k + 1) n } }{ \left( q; q \right)_{n} 
 } \left[ \begin{matrix} n - 2 \\ k - 1 \end{matrix} \right]_{q} \right).
\end{multline*}
 holds for positive integers $k$. 
\end{theorem}

\begin{proof}
 We set 
 $$ F(n, k) := (-1)^{k-1} \frac{ q^{\binom{k}{2} + (k + 1) n } }{ \left( q; q \right)_{n} 
 } \left[ \begin{matrix} n - 2 \\ k - 1 \end{matrix} \right]_{q} $$
 and $\mathcal{F}(n, k) := F(k, n)$. 
 Applying the $q$-Zeilberger algorithm 
 to $\mathcal{F}(n, k)$ yields a second-order difference equation of the desired form, 
 for the $q$-polynomials 
\begin{align*}
 p_{1}(n) & = q \left(1-q^{2 n+1}\right), \\ 
 p_{2}(n) & = -\left(1 + q^{n+1}\right) \left(1 - q^{n+2}-q^{2 n+1}+q^{2 n+4}-q^{3 n+4}+q^{4 n+4} \right), \\ 
 p_{3}(n) & = -q^{3 n+3} \left(1-q^{2 n+3}\right), 
\end{align*}
 and $\mathcal{G}(n, k) = \mathcal{R}(n, k) \mathcal{F}(n, k)$, for 
\begin{multline*}
 \mathcal{R}(n, k) = \frac{\left(1-q^k\right) \left(q^{n+1}-q^k\right)}{q \left(1-q^n\right) \left(1-q^{1+n}\right)}
 \big( q^{k+n+2}+q^{k+2 n+1}+q^{k+2 n+3} + q^{2 k+2 n+1} - \\ 
 q^{k+3 n+3}-q^{k+4 n+4}-q^{2 k}-q^k+q^{2 n+3}-q^{3 n+3}-q^{4 n+6}+q^{5 n+6} \big), 
\end{multline*}
 so that a telescoping phenomenon gives us, for 
 $f(k) = \sum_{n=2}^{\infty} F(n, k)$, that the $q$-difference equation in 
 \eqref{frecursion} holds. 
 Solving for the recurrence 
\begin{equation}\label{20727507407387077PM1A}
 p_{1}(k) g(k+2) + p_{2}(k) g(k+1) + p_{3}(k) g(k) = 0 
\end{equation} 
 for the same $q$-polynomials gives us that from \eqref{20727507407387077PM1A}, it follows that 
\begin{multline*}
 g(k) = \frac{1-q^k+q^{2 k}}{q^k} \Bigg( 
 c_{1} - c_{2} \frac{\left(1-q+q^2\right) \left(1-q^2+q^4\right)}{q^4 \left(1-q^3\right)} 
 \times \\ 
 \sum_{j=0}^{k-1} (-1)^{j} q^{ \frac{j (3 j+5)}{2} } 
 \frac{1-q^{2 j+1}}{\left(1-q^j+q^{2 j}\right) \left(1-q^{j+1}+q^{2 j+2}\right)} \Bigg) 
\end{multline*}
 for fixed $c_{1}$ and $c_{2}$. We find that 
\begin{align*}
 f(1) & = \sum_{n=2}^{\infty} \frac{ q^{2n} }{ \left( q;q \right)_{n} } 
 = \frac{1}{ \left( q^{2}; q \right)_{\infty} } - 1 - \frac{q^2}{1-q}
\end{align*}
 from a classical result in the theory of partitions 
 \cite[p.\ 19]{Andrews1998}, and we similarly find that 
\begin{align*}
 f(2) & = - \frac{q}{1-q} 
 \sum_{n=2}^{\infty} \frac{q^{3n}}{ \left( q;q \right)_{n} } \left( 1 - q^{n-2} \right) \\ 
 & = \frac{ (1-q) (1 + q) (1 - q^2 - q^3) }{q \, \left( q;q \right)_{\infty} } 
 - \frac{1 - q^2 + q^4}{ (1-q) q}. 
\end{align*}
 By then setting $f(k) = g(k)$, the above closed forms for $f(1)$ and $g(1)$
 allow us to closed for $c_{1}$ and $c_{2}$. By rewriting the 
 summand of the associated truncated sum according to the $q$-partial fraction decomposition
 $$ \frac{1}{1-q} \left( \frac{1}{1 - q^j + q^{2j}} - \frac{q}{1 - q^{j+1} + q^{2j+2}} \right) 
 = \frac{1 - q^{2j+1}}{ \left( 1 - q^j + q^{2j} \right) \left( 1 - q^{j + 1} + q^{2j+2} \right) } $$
 and by applying a reindexing argument and simplifying, we obtain an equivalent version of the desired result. 
\end{proof}

 Setting $k \to \infty$, we obtain that 
 $$ \sum_{j=0}^{k-1} (-1)^{j} q^{\frac{ j(3j-1) }{2}} \left( 1 + q^{j} \right) 
 = 1 + \left( q;q \right)_{\infty}, $$ 
 which is equivalent to the unilateral formulations of the Euler pentagonal number theorem in 
 \eqref{displayEuler} and \eqref{equivEuler}. 

 As another natural variant of the Andrews--Merca $q$-series,  we consider replacing the summand factor  $ \frac{1}{ \left( q;q 
 \right)_{n} } $ with  $\frac{1}{ \left( q;q \right)_{n + 1} }$,  noting that  $$\frac{1}{ \left( q;q \right)_{n+1} } = \frac{1}{ \left( 1 - q^{n  
 + 1} \right) \left( q;q \right)_{n} }. $$  With regard to our notation involved in Theorem \ref{anothernatural} below,  we are writing $ 
 [n]_{q} $ to denote the $q$-bracket symbol such that $[n]_{q} = \frac{1-q^n}{1-q}$. 
 A similar approach as in the proof of Theorem \ref{naturalreplace} can be employed 
 to prove the following result, 
 according to our $q$-Zeilberger-based method, by setting 
 $$ F(n, k) = (-1)^{k-1} \frac{ q^{\binom{k}{2} + (k + 1) n } }{ \left( q; q \right)_{n} 
 \left[ n + 1 \right]_{q} } \left[ \begin{matrix} n - 1 \\ k - 1 \end{matrix} \right]_{q} $$ 
 and $\mathcal{F}(n, k) := F(k, n)$ and by applying the $q$-Zeilberger 
 algorithm to $\mathcal{F}(n, k)$, yielding a second-order recurrence of the desired form. 

\begin{theorem}\label{anothernatural}
 The truncation identity 
\begin{multline*}
 \text{ {\footnotesize 
 $ q^{-k} - q^{-2k} + \frac{1 - q^{-k} + q^{-2k}}{ \left( q;q \right)_{\infty} } 
 \Bigg( \frac{ \left( -1 \right)^{k} q^{\frac{k(3k-1)}{2}} }{1 - q^{k} + q^{2k}} - 1 
 + \sum_{j=0}^{k-1} (-1)^{j} q^{\frac{j(3j-1)}{2}} \big( 1 + q^{j} \big) \Bigg) = 
 $ } }
 \\ 1 + 
 (-1)^{k-1} \sum_{n = 
 1}^{\infty} \frac{ q^{\binom{k}{2} + (k + 1) n } }{ \left( q; q \right)_{n} 
 \left[ n + 1 \right]_{q} } \left[ \begin{matrix} n - 1 \\ k - 1 \end{matrix} \right]_{q}.
\end{multline*}
 holds for positive integers $k$. 
\end{theorem}

  According to the notation in \eqref{mainobject}, Theorem \ref{naturalreplace} corresponds to the $(p(k)$, $ \ell_1$, $ \ell_2)$ $ = $ 
 $ (k + 1$, $ 2$, $ 0)$ case,  and Theorem \ref{anothernatural} corresponds to the $(p(k), \ell_1, \ell_2) = (k + 1, 1, 1)$ case. We may 
 apply our method to obtain many further $q$-series expansions for truncated versions of the $q$-series in Euler's pentagonal number 
 theorem,  and we encourage explorations of this. 
 Based on extant literature related to the Kolitsch identity and 
 the Yao identity and the Schlosser--Zhou identity 
 \cite{ChenYao2025combinatorial,ChenYao2025Proofs,ChernXia2024,HeLiu2025,Males2024,
SchlosserZhou2024,Yao2024,Yao2025Ballantine,Yao2023,Yao2022,Zhou2024}, 
 it appears that the problem of evaluating \eqref{mainobject}
 as a truncated sum for non-monic polynomials $p(k)$ has not previously been considered, 
 leading to Section \ref{sectionInfinite} below. 

\section{Infinite families of truncations}\label{sectionInfinite}
 Our method, as summarized in Section \ref{sectionNew}, 
 can be used to obtain new proofs 
 for the Andrews--Merca identity and many subsequent identities on truncated theta series, 
 including the Kolitsch identity, the Yao identity, and the 
 Schlosser--Zhou identity. This shows how our method provides a unifying framework 
 for proving and extending truncated theta series identities, 
 and this is demonstrated with our proof of the new result highlighted below. 

\begin{theorem}\label{theoreminfinite}
 For the $q$-polynomial $p(j) = 1 -q^{3 j+m}-q^{4 j+m+1}+q^{6 j + 2 m + 1}$, the truncation identity 
\begin{multline*}
 \frac{ 1 }{ \left( q^{m};q \right)_{\infty} } 
 \sum_{j = 0}^{k - 1} 
 (-1)^{j} q^{\frac{j(5j+2m-1)}{2}} \frac{ \left( q^{m}; q^2 \right)_{j} 
 \left( q^{m+1};q^2 \right)_{j} }{ \left( q;q \right)_{j} } 
 \, p(j) = \\ 
 1 + (-1)^{k - 1} \sum_{n = 
 1}^{\infty} \frac{ q^{\binom{k}{2} + 
 (2 k + m) n } }{ \left( q; q \right)_{n} } \left[ \begin{matrix} n - 1 \\ k - 1 \end{matrix} \right]_{q} 
\end{multline*}
 holds for positive integers $k$ and for $m \in \mathbb{C}$. 
\end{theorem}

\begin{proof}
 We set $$ F(n, k) := (-1)^{k-1} \frac{ q^{\binom{k}{2} + (2 k + m) n } }{ \left( q; q \right)_{n} } 
 \left[ \begin{matrix} n - 1 \\ k - 1 \end{matrix} \right]_{q} $$ 
 and $\mathcal{F}(n, k) := F(k, n)$ for a complex parameter $m$. 
 An application of the $q$-Zeilberger algorithm 
 to $\mathcal{F}(n, k)$ produces a second-order difference equation of the desired form, 
 for the $q$-polynomials 

 \ 

\noindent $p_{1}(n) = -q \big( q^{n+1}-1 \big) \big( -q^{m+3 n}-q^{m+4 n+1}+q^{2 m+6 n+1}+1 \big) $ 

 \ 

\noindent and 

 \ 

\noindent $p_{2}(n) = q \big(q^{m+3 n}-q^{2 m+6 n+1}-q^{2 m+7 n+3}-q^{2 m+8 n+5}-q^{2 m+9 n+7}+q^{3 m+9 n+3}+q^{3 m+10 n+5}+q^{3 m+10 n+6}+q^{3 m+11 n+7}+q^{3 m+11
 n+8}+q^{3 m+11 n+9}-q^{4 m+12 n+6}-q^{4 m+13 n+8}-q^{4 m+13 n+9}-q^{4 m+13 n+10}+q^{5 m+15 n+10}+q^{n+1} - 1 \big) $ 

 \ 

\noindent and 

 \ 

\noindent $p_{3}(n) = -q^{m+5 n+3} 
 \big( q^{m+2 n}-1 \big) 
 \big(q^{m+2 n+1}-1 \big) 
 \big( -q^{m+3 n+3}-q^{m+4 n+5}+q^{2 m+6 n+7} + 1 \big)$, 

 \ 

\noindent writing $\mathcal{G}(n, k) = \mathcal{R}(n, k) \mathcal{F}(n, k)$ for the $q$-rational function

 \ 

\noindent $\mathcal{R}(n, k) = \frac{1}{q^n-1} q \big(q^k-1\big) \big(q^n-q^k\big) \big(q^{k+m+3 n+2}+q^{2 k+m+3 n}+q^{2 k+m+3 n+3}+q^{3 k+m+3 n}+q^{3 k+m+3 n+4}+q^{4 k+m+3
 n}+q^{4 k+m+4 n+1}-q^{k+m+5 n+2}-q^{2 k+m+5 n+2}-q^{3 k+m+5 n+2}-q^{k+2 m+6 n+2}-q^{k+2 m+6 n+5}-q^{2 k+2 m+6 n+1}-q^{2 k+2 m+6 n+3} - $

\noindent $ q^{2 k+2 m+6 n+7}-q^{3 k+2 m+6 n+1}-q^{3 k+2 m+6 n+4}-q^{4 k+2 m+6 n+1}+q^{k+2 m+7 n+2} - $

\noindent $ q^{k+2 m+7 n+7}+q^{2 k+2 m+7 n+2}-q^{2 k+2 m+7 n+4}+q^{3 k+2
 m+7 n+2}-q^{3 k+2 m+7 n+5} + $

\noindent $ q^{k+2 m+8 n+5}+q^{k+2 m+9 n+7}+q^{k+3 m+9 n+5}+q^{k+3 m+9 n+9}+q^{2 k+3 m+9 n+4} + $ 

\noindent $ q^{2 k+3 m+9 n+7}+q^{3 k+3 m+9 
 n+5}-q^{k+3 m+10 n+5}+q^{k+3 m+10 n+7}+q^{2 k+3 m+10 n+8}-q^{k+3 m+11 n+7}-q^{k+3 m+11 n+9}-q^{k+4 m+12 n+9}-q^{2 k+4 m+12 n+8}+q^{k+4 m+13
 n+9}+q^{2 k+n+1}+q^{3 k+n+1}-q^{2 k}-q^{3 k}-q^{4 k}+q^{m+4 n+2}-q^{m+5 n+2}-q^{2 m+6 n+3}+q^{2 m+7 n+3}-q^{2 m+7 n+5}+q^{2 m+8 n+5}-q^{2 m+8
 n+7}+q^{2 m+9 n+7}+q^{3 m+9 n+6}-q^{3 m+10 n+6}+q^{3 m+10 n+8}+q^{3 m+10 n+9}-q^{3 m+11 n+8}-q^{3 m+11 n+9}-q^{4 m+12 n+10}+q^{4 m+13
 n+10}\big). $

 \ 

\noindent We thus find that the recursion in \eqref{frecursion} holds for the specified   $q$-polynomials and for $f(n) = \sum_{k = 
  1}^{\infty} F(n, k)$, with the base cases  holding from classically known relations for basic hypergeometric series.  
\end{proof}

 Setting $k \to \infty$ in Theorem \ref{theoreminfinite}, we obtain that 
 $$ \text{ {\footnotesize $ \sum_{j = 0}^{\infty} 
 (-1)^{j} q^{\frac{j(5j+2m-1)}{2}} \frac{ \left[ \begin{matrix} q^{m}, q^{m+1} 
 \end{matrix} \, \big| \, q^2 \right]_{j} }{ \left( q;q \right)_{j} } 
 \big( 1 -q^{3 j+m}-q^{4 j+m+1}+q^{6 j + 2 m + 1} \big) 
 = \left( q^{m};q \right)_{\infty},$ } } $$ writing 
\begin{equation*}
 \left[ \begin{matrix} \alpha, \beta, \ldots, \gamma 
 \end{matrix} \, \big| \, q \right]_{n} 
 = \left( \alpha; q \right)_{n} 
 \left( \beta; q \right)_{n} \cdots \left( \gamma; q \right)_{n}.
\end{equation*}
 A similar approach, as in our proof of Theorem \ref{theoreminfinite}, 
 allows us to extend the Andrews--Merca identity, 
 by obtaining explicit, finite $q$-summations 
 for \eqref{mainobject} if $\rho(k)$ is linear with a fixed leading term. 
 In this direction, 
 for the $\rho(k) = 3 k +m$, for a free, complex parameter $m$, 
 and for the $q$-polynomial 

 \ 

\noindent $p(j) = 1 -q^{4 j+m}-q^{5 j+m+1}-q^{6 j+m+2}+q^{8 j+2 m+1} + q^{9 j+2 m+2} + q^{9 j+2 m+3}-q^{12 j+3 m+3}$, 

 \ 

\noindent we obtain that 
\begin{multline*}
 \frac{ 1 }{ \left( q^{m};q \right)_{\infty} } 
 \sum_{j = 0}^{k - 1} 
 (-1)^{j} q^{ \frac{j (7 j+2 m-1)}{2} } 
 \frac{ \left[ \begin{matrix} q^{m}, q^{m+1}, q^{m+2} \end{matrix} \, \big| \, q^3 \right]_{j} }{ \left( q;q \right)_{j} } 
 \, p(j) = \\ 
 1 + (-1)^{k-1} \sum_{n = 1}^{\infty} \frac{ q^{\binom{k}{2} + (3 k + m) n } }{ \left( q; q \right)_{n} } 
 \left[ \begin{matrix} n - 1 \\ k - 1 \end{matrix} \right]_{q}. 
\end{multline*}
 Similarly, for 

 \ 

\noindent $p(j) = 1 -q^{5 j+m} - 
 q^{6 j+m+1}-q^{7 j+m+2} - 
 q^{8 j+m+3}+q^{10 j+2 m+1} + 
 q^{11 j+2 m+2} +q^{11 j+2 m+3} + 
 q^{12 j+2 m+3}+q^{12 j+2 m+4}+q^{12 j+2 m+5}-q^{15 
 j+3 m+3}-q^{16 j+3 m+4} - $ 

\noindent $q^{16 j+3 m+5}-q^{16 j+3 m+6}+q^{20 j+4 m+6}$, 

 \ 

\noindent we obtain the truncation identity such that 
\begin{multline*}
 \frac{ 1 }{ \left( q^{m};q \right)_{\infty} } 
 \sum_{j = 0}^{k - 1} 
 (-1)^{j} q^{ \frac{j (9 j+2 m-1)}{2} } 
 \frac{ \left[ \begin{matrix} q^{m}, q^{m+1}, q^{m+2}, q^{m+3} 
 \end{matrix} \, \big| \, q^4 \right]_{j} }{ \left( q;q \right)_{j} } 
 \, p(j) = \\ 
 1 + (-1)^{k-1} \sum_{n = 1}^{\infty} \frac{ q^{\binom{k}{2} + (4 k + m) n } }{ \left( q; q \right)_{n} } 
 \left[ \begin{matrix} n - 1 \\ k - 1 \end{matrix} \right]_{q}. 
\end{multline*}

\section{Conclusion}
 Our method is broadly applicable when it comes to proving and generating truncation identities for $q$-series, and we encourage 
 further explorations of this method. Also, we leave it to a separate project to obtain and apply combinatorial interpretations based on 
 truncation identities obtained from our method. 

\subsection*{Acknowledgements}
 The author was supported through a Killam Postdoctoral Fellowship from the Killam Trusts and thanks Karl Dilcher, Lin Jiu, and Shane 
 Chern for useful comments related to this paper.

 \ 

{\textsc{John M. Campbell}} 

\vspace{0.1in}

Department of Mathematics and Statistics 

Dalhousie University 

6299 South St, Halifax, NS B3H 4R2

\vspace{0.1in}

{\tt jh241966@dal.ca}

\end{document}